\documentclass[reqno]{amsart}
\usepackage{amssymb}
\usepackage[mathscr]{euscript}
\usepackage{amsmath}
\makeatletter
\@addtoreset{equation}{section}
\makeatother

\renewcommand\thefigure{\thesection.\@arabic\c@figure}
\renewcommand\thetable{\thesection.\@arabic\c@table}

\newtheorem{theorem}{Theorem}[section]
\newtheorem{lemma}[theorem]{Lemma}
\newtheorem{proposition}[theorem]{Proposition}

\newtheorem{remark}[theorem]{Remark}

\newcommand{\mc}[1]{{\mathcal #1}}
\newcommand{\ms}[1]{{\mathscr #1}}
\newcommand{\mf}[1]{{\mathfrak #1}}

\newcommand{\bb}[1]{{\mathbb #1}}
\newcommand{\bs}[1]{{\boldsymbol #1}}
\newcommand{\mce}[1]{{\mathscr E_N^{#1}}}

\renewcommand{\Cap}{{\rm cap}}

\begin{document}

\title[]{Metastability of reversible condensed zero range processes on
  a finite set}

\author{J. Beltr\'an, C. Landim}

\address{\noindent Institut de Math\'ematiques, \'Ecole Polytechnique
  F\'ed\'erale de Lausanne, Station 8, CH-1015 Lausanne, Switzerland.
  and PUCP, Av. Universitaria cdra. 18, San Miguel, Ap. 1761, Lima
  100, Per\'u.  \newline e-mail: \rm \texttt{johel@impa.br} }

\address{\noindent IMPA, Estrada Dona Castorina 110, CEP 22460 Rio de
  Janeiro, Brasil and CNRS UMR 6085, Universit\'e de Rouen, Avenue de
  l'Universit\'e, BP.12, Technop\^ole du Madril\-let, F76801
  Saint-\'Etienne-du-Rouvray, France.  \newline e-mail: \rm
  \texttt{landim@impa.br} }


\begin{abstract}
  Let $r: S\times S\to \bb R_+$ be the jump rates of an irreducible
  random walk on a finite set $S$, reversible with respect to some
  probability measure $m$. For $\alpha >1$, let $g: \bb N\to \bb R_+$
  be given by $g(0)=0$, $g(1)=1$, $g(k) = (k/k-1)^\alpha$, $k\ge
  2$. Consider a zero range process on $S$ in which a particle jumps
  from a site $x$, occupied by $k$ particles, to a site $y$ at rate
  $g(k) r(x,y)$. Let $N$ stand for the total number of particles. In
  the stationary state, as $N\uparrow\infty$, all particles but a
  finite number accumulate on one single site. We show in this article
  that in the time scale $N^{1+\alpha}$ the site which concentrates
  almost all particles evolves as a random walk on $S$ whose
  transition rates are proportional to the capacities of the
  underlying random walk.
\end{abstract}

\maketitle

\section{Introduction}
\label{sec-1} 

Fix a finite state space $S$ of cardinality $\kappa\ge 2$ and consider
an irreducible continuous time random walk $\{X_t : t\ge 0\}$ on $S$
which jumps from $x$ to $y$ at some rate $r(x,y)$. Assume that this
dynamics is reversible with respect to some probability measure $m$ on
$S$: $m(x) r(x,y) = m(y) r(y,x)$, $x$, $y \in S$. Denote by $\Cap_S$
the capacity associated to this random walk: For two disjoint proper
subsets $A$, $B$ of $S$,
\begin{equation*}
\Cap_S(A,B) \;=\; \inf_{f\in \mc B(A,B)} \frac 12 \sum_{x,y\in S} m(x)
\, r(x,y) \, \{f(y) - f(x)\}^2\; ,
\end{equation*}
where $\mc B(A,B)$ stands for the set of functions $f:S\to \bb R$
equal to $1$ at $A$ and equal to $0$ at $B$. When $A = \{x\}$,
$B=\{y\}$, we represent $\Cap_S(A,B)$ by $\Cap_S(x,y)$.

Let $M_\star$ be the maximum value of the probability measure $m$:
$M_\star = \max \{m(x) : x\in S\}$ and denote by $S_\star$ the sites
where $m$ attains its maximum value: $S_\star = \{x \in S : m(x) =
M_\star\}$. Of course, in the symmetric, nearest--neighbor case, where
$r(x,y)=1$ if $y = x \pm 1$, modulo $\kappa$, and $r(x,y)=0$
otherwise, $m$ is constant and $S_\star$ and $S$ coincide.

Fix a real number $\alpha>1$. Let $g:\bb N\to \bb R$ be given by
$$
g(0)=0\; , \quad g(1)=1\;, \quad
\textrm{and}\quad g(n) \;=\; \Big( \frac{n}{n-1} \Big)^\alpha\;,
\; n\ge 2\;,
$$
so that $\prod_{i=1}^{n}g(i)=n^{\alpha}$, $n\ge 1$. Consider the
zero range process on $S$ in which a particle jumps from a site $x$,
occupied by $k$ particles, to a site $y$ at rate $g(k) r(x,y)$. Since
$g$ is decreasing, the dynamics is attractive in the sense that
particles on sites with a large number of particles leave them at a
slower rate than particles on sites with a small number of particles.

The total number of particles is conserved by the dynamics, and for
each fixed integer $N\ge 1$ the process restricted to the set of
configurations with $N$ particles, denoted by $E_N$, is
irreducible. Let $\mu_N$ be the unique invariant probability measure
on $E_N$. When $\alpha>2$, the measure $\mu_N$ exhibits a very
peculiar structure called condensation in the physics literature.
Mathematically, this means that under the stationary state, above a
certain critical density, as the total number of particles
$N\uparrow\infty$, only a finite number of particles are located on
the sites which do not contain the largest number of particles.

Condensation has been observed and investigated in shaken granular
systems, growing and rewiring networks, traffic flows and wealth
condensation in macroeconomics. We refer to the recent review by Evans
and Hanney \cite{eh}.

Several aspects of the condensation phenomenon for zero range dynamics
have been examined. Let the condensate be the site with the maximal
occupancy.  Precise estimates on the number of particles at the
condensated, as well as its fluctuations, have been obtained in
\cite{jmp, gss, emz06}. The equivalence of ensembles has been proved
by Gro{\ss}kinsky, Sch{\"u}tz and Spohn \cite{gss}.  Ferrari, Landim
and Sisko \cite{fls} proved that if the number of sites is kept fixed,
as the total number of particles $N\uparrow\infty$, the distribution
of particles outside the condensated converges to the grand canonical
distribution with critical density. Armendariz and Loulakis \cite{al}
generalized this result showing that if the number of sites $\kappa$
grows with the number of particles $N$ in such a way that the density
$N/\kappa$ converges to a value greater than the critical density, the
distribution of the particles outside the condensate converges to the
grand canonical distribution with critical density.

We investigate in this article the dynamical aspects of the
condensation phenomenon. Fix an initial configuration with the
majority of particles located at one site. Denote by $X^N_t$ the
position of the condensated at time $t\ge 0$. In case of ties, $X^N_t$
remains in the last position. The process $\{X^N_{t} : t\ge 0\}$
evolves randomly on $S$ according to some non-Markovian dynamics.

The main result of this article states that, for $\alpha>1$, in the
time scale $N^{1+\alpha}$, the process $\{X^N_t : t\ge 0\}$ evolves
asymptotically according to a random walk on $S_\star$ which jumps
from $x$ to $y$ at a rate proportional to the capacity $\Cap_S(x,y)$.
In the terminology of \cite{bl2}, we are proving that the condensate
exhibits a tunneling behavior in the time scale $N^{1+\alpha}$.

This article leaves two interesting open questions. The techniques
used here rely strongly on the reversibility of the process. It is
quite natural to examine the same problem for asymmetric zero range
processes where new techniques are required. On the other hand, the
number of sites is kept fixed. It is quite tempting to let the
number of sites grow with the number of particles. In this case, in
the nearest neighbor, symmetric model, for instance, the condensate
jumps from one site to another at rate proportional to the inverse of
the distance. The rates are therefore not summable and it is not clear
if a scaling limit exists.

Simulations for the evolution of the condensated have been performed
by Go\-dr\`e\-che and Luck \cite{gl}. The authors predicted the time
scale, obtained here, in which the condensate evolves and claimed that
the time scale should be the same for non reversible dynamics.

\section{Notation and results}
\label{sec0}

Throughout this article we fix a finite set $S$ of cardinality
$\kappa\ge 2$ and a real number $\alpha>1$. For each $S_0\subseteq S$
consider the set of configurations $E_{N,S_0}$, $N\ge 1$, given by
$$
E_{N, S_0}\;:=\;\big\{ \eta\in\bb N^{S_0} : \sum_{x\in S_0} \eta_x = N \big\}\;,
$$
where $\bb N=\{0,1,2,...\}$. When $S_0=S$, we use the shorthand $E_N$
for $E_{N,S}$. Define $a(n)=n^{\alpha}$ for $n\ge 1$ and set
$a(0)=1$. Let us also define $g:\bb N\to \bb R_+$,
$$
g(0)=0\; , \quad g(1)=1  \quad
\textrm{and}\quad g(n)=\frac{a(n)}{a(n-1)}\;, \;\; n\ge 2\;,
$$
in such a way that $\prod_{i=1}^{n}g(i)=a(n)$, $n\ge 1$, and $\{ g(n)
: n\ge 2\}$ is a strictly decreasing sequence converging to $1$ as
$n\uparrow\infty$.

Consider a random walk on $S$ with jump rates denoted by $r$. Its
generator $\mc L_S$ acts on functions $f:S\to \bb R$ as
$$
(\mc L_S f)(x) \;=\; \sum_{y\in S} r(x,y) \{f(y) - f(x) \}\;.
$$
Assume that this Markov process is irreducible and \emph{reversible}
with respect to some probability measure $m$ on $S$:
\begin{equation}
\label{f01}
m(x) \, r(x,y) \;=\; m(y)\, r(y,x)\;, \quad x,y \in S\;.
\end{equation}
Let $M_\star$ be the maximum value of the probability measure $m$, let
$S_\star \subset S$ be the sites of $S$ where $m$ attains its maximum
and let $\kappa_\star$ be the cardinality of $S_\star$:
\begin{equation*}
\begin{split}
& M_\star = \max \{ m(x) : x\in S\}\;,\quad 
S_\star = \{x \in S : m(x) = M_\star\} \quad \textrm{and}\quad
\kappa_\star \;=\; | S_\star | \;.
\end{split}
\end{equation*}
In addition, let $m_{\star}(x)=m(x)/M_{\star}$ so that
$m_{\star}(x)=1$ for any $x\in S_{\star}$.  Denote by $D_S$ the
Dirichlet form associated to the random walk:
\begin{equation}
\label{f14}
D_S(f) \;=\; \frac 12 \sum_{x, y \in S} m(x) r(x,y) \{f(y) - f(x) \}^2 
\end{equation}
for $f:S\to \bb R$, and denote by $\Cap_S(x,y)$ the capacity between
two different points $x,y\in S$:
\begin{equation}
\label{ff1}
\Cap_S(x,y) \;=\; \inf_{f\in \mc B(x,y)} \{ D_S(f) \}\;,
\end{equation}
where the infimum is carried over the set $\mc B(x,y)$ of all
functions $f:S\to \bb R$ such that $f(x) = 1$ and $f(y)=0$. \smallskip

For each pair $x,y\in S, x\not = y$, and $\eta\in E_N$ such that
$\eta_x>0$, denote by $\sigma^{xy}\eta$ the configuration obtained from
$\eta$ by moving a particle from $x$ to $y$:
$$
(\sigma^{xy}\eta)_z\;=\;\left\{
\begin{array}{ll}
\eta_x-1 & \textrm{for $z=x$} \\
\eta_y+1 & \textrm{for $z=y$} \\
\eta_z & \rm{otherwise}\;. \\
\end{array}
\right.
$$

For each $N\ge 1$, consider the zero range process defined as the
Markov process $\{\eta^N(t) : t\ge 0\}$ on $E_N$ whose generator $L_N$
acts on functions $F: E_N\to\bb R$ as
\begin{equation}
\label{f16}
(L_N F) (\eta) \;=\; \sum_{\stackrel{x,y\in S}{x\not = y}}
g(\eta_x) \, r(x,y) \, \big\{ F(\sigma^{xy}\eta) - F(\eta) \big\} \;. 
\end{equation}
The Markov process corresponding to $L_N$, $N\ge 1$, is irreducible
and reversible with respect to its unique invariant measure $\mu_N$
given by
$$
\mu_N(\eta) \;=\; \frac {N^{\alpha}} {Z_{N, S}} \, 
\prod_{x\in S} \frac{m_\star(x)^{\eta_x}}{ a(\eta_x)}
\;:=\;  \frac{N^{\alpha}}{Z_{N,S}} \, \frac{m_{\star}^\eta}
{a(\eta)} \;, \quad \eta\in E_N \;,
$$
where, for any $S_0\subseteq S$,
$$
m_\star^{\zeta} = \prod_{x\in S_0} m_\star(x)^{\zeta_x}\;, \quad
a(\zeta)\;=\;\prod_{x\in S_0}a(\zeta_x)\;, \quad \zeta\in {\bb
  N}^{S_0}\;, 
$$
and $Z_{N, S_0}$ is the normalizing constant
\begin{equation}
\label{f03}
Z_{N, S_0}\;=\; N^{\alpha} 
\sum_{\zeta\in E_{N,S_0}} \frac{m_\star^{\zeta}}{a(\zeta)}\;\cdot
\end{equation}

In Section \ref{sec2} we show that the sequence $\{Z_{N, S} : N\ge
1\}$ converges as $N\uparrow\infty$. This explains the factor
$N^\alpha$ in its definition. The precise statement is as follows. For
$x$ in $S$ and $\kappa\ge 2$, let
\begin{equation*}
\Gamma_x \;:=\; \sum_{j\ge 0} \frac{m_\star(x)^j}{a(j)}\;, \quad
\Gamma(\alpha) \;:=\; \sum_{j\ge 0} \frac{1}{a(j)}\;
\end{equation*}
so that $\Gamma(\alpha)=\Gamma_x$ for any $x\in S_{\star}$, and define
\begin{equation*}
Z_S \;:=\; \frac{\kappa_{\star}}{\Gamma(\alpha)} 
\prod_{z\in S} \Gamma_z \;=\; \kappa_{\star} \Gamma(\alpha)^{\kappa_\star -1}
\prod_{y\not\in S_\star} \Gamma_y\;.
\end{equation*}
\begin{proposition}
\label{zk}
For every $\kappa\ge 2$,
\begin{equation*}
\lim_{N\to\infty} Z_{N,S} = Z_{S}\;.
\end{equation*}
\end{proposition}

Denote by $D_N$ the Dirichlet form associated to the generator
$L_N$. An elementary computation shows that
\begin{equation*}
D_N(F)\;=\; \frac 12 \sum_{\stackrel{x,y\in S}{x\not = y}}
\sum_{\eta\in E_N} \mu_N(\eta)\, 
g(\eta_x) \, r(x,y) \, \{ F(\sigma^{xy}\eta) - F(\eta) \}^2 \;,
\end{equation*}
for every $F:E_N\to\bb R$.

For every two disjoint subsets $A,B$ of $E_N$ denote by
$\mc C_N(A,B)$ the set of functions $F:E_N\to\bb R$ defined by  
$$
\mathcal C_N(A,B) \;:=\; \{ F : \textrm{$F(\eta)=1$ 
$\forall$ $\eta\in A$ and $F(\xi)=0$ $\forall$ $\xi\in B$}\}\;.
$$
The capacity corresponding to this pair of disjoint subsets $A,B$ is
defined as
$$
\Cap_N(A,B) \;:=\; \inf\big\{\,D_N(F) : F\in \mathcal C_N(A,B)\,\big\}\;.
$$
Since $D_N(F)=D_N(1-F)$, $\Cap_N(A,B) = \Cap_N(B,A).$

Fix a sequence $\{\ell_N : N\ge 1\}$ such that $1\ll \ell_N \ll N$
and, for each $z\in S\setminus S_{\star}$, fix a sequence $\{b_N(z) :
N\ge 1\}$ such that $1\ll b_N(z)$:
\begin{equation}
\label{f18}
\lim_{N\to\infty} \ell_N \;=\; \infty\;, \quad 
\lim_{N\to\infty} \ell_N/N \;=\; 0\;, \quad \textrm{and}\quad
\lim_{N\to\infty} b_N(z) \;=\; \infty\;,
\end{equation}
for all $z\in S\setminus S_{\star}$. For $x$ in $S_\star$, let
\begin{equation*}
\ms E^x_N  \;:=\; \Big\{\eta\in E_N : \eta_x \ge N - \ell_N \,,\,
\eta_z \le b_N(z)\,,\, z\not\in S_\star \Big\}\;.
\end{equation*}
Obviously, $\ms E^x_N\not = \varnothing$ for all $x\in S_{\star}$ and
every $N$ large enough. In the case where the measure $m$ is uniform,
the second condition is meaningless and the set $\ms E^x_N$ becomes
$\ms E^x_N = \{\eta\in E_N : \eta_x \ge N - \ell_N\}$. 

Condition $\ell_N/N\to 0$ is required to guarantee that on each set
$\ms E^x_N$ the proportion of particles at $x\in S_\star$,
i.e. $\eta_x/N$, is almost one. As a consequence, for $N$ sufficiently
large, the subsets $\ms E^x_N$, $x\in S_\star$, are pairwise disjoint.
>From now on, we assume that $N$ is large enough so that the partition
\begin{equation}
\label{f15}
E_N\;=\; \Big( \bigcup_{x\in S_\star}\mce x \Big) \cup \Delta_N
\end{equation}
is well defined, where $\Delta_N$ is the set of configurations which
do not belong to any set $\ms E^x_N$, $x\in S_\star$.  

The assumptions that $\ell_N\uparrow \infty$ and that $b_N(z)\uparrow
\infty$ for all $z\not \in S_{\star}$ are sufficient to prove that
$\mu_N(\Delta_N)\to 0$, as we shall see in (\ref{delta}), and to
deduce the limit of the capacities stated in Theorem \ref{mt1}
below. In particular, in these two statements we may set $b_N(z)=N$,
$z\not \in S_\star$, in order to discard the second restriction in the
definition of the sets $\ms E^x_N$, $x\in S_\star$. We need, however,
further restrictions on the growth of $\ell_N$ and $b_N(z)$ to prove
the tunneling behaviour of the zero range processes presented in
Theorem \ref{mt2} below.

To state the first main result of this article, for any nonempty
subset $S_\star'$ of $S_\star$, let $\mce{}(S_\star')\,=\, \cup_{x\in
  S_\star'} \mce x$, and let
\begin{equation}
\label{defi}
I_\alpha \;:=\; \int_{0}^1 u^\alpha(1-u)^\alpha \, du\;.
\end{equation}

\begin{theorem}
\label{mt1}
Assume that $\kappa_\star\ge 2$. Fix a nonempty subset $S_\star^1
\subsetneq S_\star$ and denote $S_\star^2= S_\star\setminus
S_\star^1$. Then,
\begin{equation*}
\lim_{N\to\infty} N^{1+\alpha} \Cap_N\big(\ms E_N(S_\star^1), \ms
E_N({S_\star^2})\big) \;=\; \frac 1{M_\star \, \kappa_\star \,
  \Gamma(\alpha) \, I_\alpha} 
\sum_{x \in S_\star^1, y\in S_\star^2} \Cap_S(x,y) \; .
\end{equation*}
\end{theorem}

Note that the right hand side depends on the sites not in $S_\star$
through the capacities $\Cap_S(x,y)$. In the case where the
measure $m$ is constant, $S_\star = S$, 
$M_\star = \kappa_\star^{-1}$ and the right hand side becomes
\begin{equation*}
\frac 1{ \Gamma(\alpha) \, I_\alpha } 
\sum_{x \in S_\star^1, y\in S_\star^2} \Cap_S(x,y) \;.
\end{equation*}

To prove Theorem \ref{mt1}, we derive a lower and an upper bound for
the capacity. In the first part, we need to obtain a lower bound for
the Dirichlet form of functions in $\mc C_N(\ms E_N(S_\star^1), \ms
E_N( S_\star^2))$. To our advantage, since it is a lower bound, we may
neglect some bonds in the Dirichlet form we believe to be
irrelevant. On the other hand, and this is the main difficulty, the
estimate must be uniform over $\mc C_N(\ms E_N(S_\star^1), \ms E_N(
S_\star^2))$. As we shall see in Section \ref{sec3}, the proof of a
sharp lower bound gives a clear indication of the qualitative behavior
of the function which solves the variational problem appearing in the
definition of the capacity. With this information, we may propose a
candidate for the upper bound. Here, in contrast with the first part,
we have to estimate the Dirichlet form of a specific function, our
elected candidate, but we need to estimate all the Dirichlet form and
can not neglect any bond.

For each $\eta\in E_N$, let ${\bf P}^N_{\eta}$ stand for the
probability on the path space $D(\bb R_+, E_N)$ induced by the zero
range process $\{\eta^N(t) : t\ge 0\}$ introduced in (\ref{f16})
starting from $\eta\in E_N$. Expectation with respect to ${\bf
  P}^N_{\eta}$ is denoted by ${\bf E}^N_{\eta}$. In addition, for any
$A\subseteq E_N$, let $T_A$ denote the hitting time of $A$:
$$
T_A \;:=\;\inf \big\{ t\ge 0 : \eta^N(t) \in A \big\}\;.
$$
 
\begin{remark}
  It is well known (see e.g. Lemma 6.4 in \cite{bl2}) that the
  solution of the variational problem for the capacity is given by
\begin{equation*}
{\bf F_{S_\star^1, S_\star^2}} (\eta) \;=\; {\bf P}^N_{\eta} \Big[
T_{\ms E_N(S_\star^1)} < T_{\ms E_N(S_\star^2)} \Big] \;.
\end{equation*}
The candidate proposed in the proof of the upper bound provides,
therefore, an approximation, in the Dirichlet sense, of the function
${\bf F_{S_\star^1, S_\star^2}}$.
\end{remark}

The second main result of this article states that the zero range
process exhibits a metastable behavior as defined in \cite{bl2}. Fix a
nonempty subset $A$ of $E_N$. For each $t\ge 0$, let $\mc T^{A}_t$ be
the time spent by the zero range process $\{\eta^N(t) : t\ge 0\}$
on the set $A$ in the time interval $[0,t]$:
$$
\mc T^{A}_t \;:=\;\int_{0}^t \mathbf{1}\{\eta^N(s) \in A\} \,ds
$$
and let $\mc S^{A}_t$ be the generalized inverse of $\mc T^{A}_t:$
$$
\mc S^{A}_t \;:=\;\sup\{s\ge 0 : \mc T^{A}_s \le t\}\;.
$$
It is well known that the process $\{\eta^{N,A}(t) : t\ge 0\}$ defined
by $\eta^{N,A}(t) = \eta^N({\mc S^{A}_t})$ is a strong Markov process
with state space $A$ \cite{bl2}. This Markov process is called
the trace of the Markov process $\{\eta^N(t) : t\ge 0\}$ on $A$.

Consider the trace of $\{\eta^N(t) : t\ge 0\}$ on $\ms E_N(S_\star)$,
referred to as $\eta^{\ms E_N^\star}(t)$. Let $\Psi_N:\ms
E_N(S_\star)\mapsto S_\star$ be given by
$$
\Psi_N(\eta) \;=\; \sum_{x\in S_\star} x\, \mathbf 1\{\eta \in \ms
E^x_N\} 
$$
and let $X^N_t:=\Psi_N(\eta^{\ms E_N^\star}(t))$.

We prove in Theorem \ref{mt2} below that the speeded up non-Markovian
process $\{X^N_{tN^{\alpha+1}} : t\ge 0\}$ converges to the random
walk $\{X_t : t\ge 0\}$ on $S_\star$ whose generator $\bb L_{S_\star}$
is given by
\begin{equation}
\label{f17}
(\bb L_{S_\star} f) (x) \;=\;  \frac 1 { M_\star
\, \Gamma(\alpha) \, I_{\alpha} } 
\sum_{y\in S_\star} \Cap_S(x,y)\, \{f(y) - f(x) \}\;.
\end{equation}
For $x$ in $S_\star$, denote by $\bb P_x$ the probability measure on
the path space $D(\bb R_+, S_\star)$ induced by the random walk $\{X_t
: t\ge 0\}$ starting from $x$.

\renewcommand{\theenumi}{\Alph{enumi}}
\renewcommand{\labelenumi}{(\theenumi)}

\begin{theorem}
\label{mt2}
Assume that $\kappa_\star\ge 2$. If \eqref{f18} holds and
\begin{equation}\label{clb}
\lim_{N\to\infty}\frac{\ell_N^{\, 1+ \alpha(\kappa-1)}}{
N^{1+\alpha}} \, \prod_{z\in S\setminus S_\star} m_\star(z)^{-b_N(z)} \;=\;0
\end{equation}
then, for each $x\in S_\star$,
\begin{enumerate}
\item[({\bf M1})]
We have
\begin{equation*}
\lim_{N\to\infty} \inf_{\eta,\xi\in \ms E^x_N} {\bf P}^N_{\eta}
\big[\,T_{\{ \xi\}} < T_{{\ms E}_N(S_\star\setminus \{x\})}\,\big] \;=\; 1 \;;
\end{equation*}

\item[({\bf M2})] For any sequence $\xi_N\in \ms E^x_N$, $N\ge 1$, the
  law of the stochastic process $\{X^N_{tN^{\alpha +1}} : t\ge 0\}$
  under ${\bf P}^N_{\xi_N}$ converges to $\bb P_{x}$ as
  $N\uparrow\infty$;

\item[({\bf M3})]
For every $T> 0$,
$$
\lim_{N\to\infty} \sup_{\eta\in \ms E^x_N} {\bf E}^N_{\eta} \Big[\, 
\int_0^T {\bs 1}\big\{ \eta^N(sN^{\alpha +1})\in \Delta_N \big\}\, ds \,\Big]\;
=\; 0 \;.
$$
\end{enumerate}
\end{theorem}

If $\kappa > \kappa_\star$, in order to fulfill conditions (\ref{f18})
and (\ref{clb}), we can take, for instance,
$$
b_N(z)\;=\; \frac{ - \log(\ell_N)}{\log(m_\star(z))} \quad
\textrm{for $z\in S\setminus S_\star$} \quad \textrm{and} \quad \ell_N
\;=\; N^{1/(\kappa -1)}
$$
if $\kappa_\star \ge 3$, $\ell_N = N^{1/[\kappa -(1/2)]}$ if
$\kappa_\star =2$.

According to the terminology introduced in \cite{bl2}, Theorem
\ref{mt2} states that the sequence of zero range processes
$\{\eta^{N}(t) : t\ge 0\}$ exhibits a tunneling behaviour on the
time-scale $N^{\alpha + 1}$ with metastates given by $\{\ms E^x_N :
x\in S_\star \}$ and limit given by the random walk $\{X_t : t\ge
0\}$.

Property ({\bf M3}) states that, outside a time set of order smaller
than $N^{\alpha +1}$, one of the sites in $S_\star$ is occupied by at
least $N-\ell_N$ particles. Property ({\bf M2}) describes the
time-evolution on the scale $N^{\alpha +1}$ of the site concentrating
the largest number of particles. It evolves asymptotically as a Markov
process on $S_\star$ which jumps from a site $x$ to $y$ at a rate
proportional to the capacity $\Cap_S(x,y)$ of the underlying random
walk. Property ({\bf M1}) guarantees that the process starting in a
metastate $\ms E^x_N$ thermalizes therein before reaching any other
metastate.

\begin{remark}
In \cite{gss}, it is shown that, in the case the number of sites increases with
the number of particles, the highest occupied
site contains a nonzero fraction of the particles in the system. This result includes the case $1<\alpha \le 2$. In contrast, when the number of sites is kept fixed, it seems to have been unnoticed in the literature that the condensation phenomenon appears also for $1<\alpha \le 2$. More precisely, if $1 \ll \ell_N \ll N$, then
\begin{equation*}
\lim_{N \to\infty}\mu_N \big( \eta_x \ge N - \ell_N \big) 
\;=\; 1/\kappa_\star \;, \quad \forall x\in S_\star\;.
\end{equation*}
Moreover, given that particles concentrate on $x\in S_\star$, the distribution of the
configuration on $S\setminus \{x\}$ is asymptotically given by the
grand-canonical measure determined by $m_\star$: For any $x$ in $S_\star$,
\begin{equation*}
\lim_{N \to\infty}\sup_{\zeta\in \mc G^x_N} \Big| \, 
\mu_N \big( \eta_z = \zeta_z \, , \, z\not = x
\,\big| \,  \eta_x\ge N - \ell_N \big) \,-\, \prod_{z\not =x} \frac 1{\Gamma_z}
\, \frac {m_\star(z)^{\zeta_z}}{a(\zeta_z)} \, \Big| \;=\; 0 \;,
\end{equation*}
where $\ms G^x_N := \{\zeta \in {\bb N}^{S\setminus\{x\}} : \sum_{z}\zeta_z \le \ell_N\}$.
There is just a small difference between the cases $1<\alpha \le 2$
and $\alpha>2$. While in the former, the variables $\{\eta_z : z\in
S_\star\}$ do not have finite expectation under the critical
grand-canonical measure, they do have finite expectation in the latter
case.
\end{remark}

In \cite{bl3}, we have proved Theorems \ref{mt1} and \ref{mt2} in the
case where the rates $r(\cdot, \cdot)$ in the definition of the
sequence of zero range processes corresponds to a random walk on a
finite complete graph. Since it covers the case $\kappa=2$, we may
suppose that $\kappa \ge 3$.

\section{The stationary measure $\mu_N$}
\label{sec2} 

In this section, we prove Proposition \ref{zk}. The proof relies on
four lemmata. We first show that the sequence $Z_{N,S}$ is bounded
below by a strictly positive constant and above by a finite
constant. Let
\begin{equation*}
\tilde Z_{N,\kappa} \;=\; N^{\alpha} 
\sum_{\eta\in E_N} \frac{1}{a(\eta)}
\end{equation*}
and note that $Z_{N,S} \le \tilde Z_{N,|S|}$.

\begin{lemma}
\label{s02}
For each $\kappa\ge 2$, there exists a constant $A_{\kappa}>0$, which
only depends on $\alpha$ and $\kappa$, such that 
\begin{equation*}
1 \;\le\; Z_{N,S} \;\le\; \tilde Z_{N,\kappa} \;\le\;
A_\kappa\;. 
\end{equation*}
\end{lemma}

\begin{proof}
Choose $x$ in $S_\star$ and denote by $\xi$ the
configuration in $E_N$ such that $\xi(x)=N$, $\xi(y)=0$ for $y\not
=x$.  By definition, $m_\star(x) =1$ so that $Z_{N,S} \ge
N^\alpha  m_\star^{\xi} / a(\xi) = 1$, which
proves the lower bound.

We proceed by induction to prove the upper bound. The estimate clearly
holds for $\kappa=2$. Assume that it is in force for $2\le \kappa
<k$. The identity
\begin{equation*}
\tilde Z_{N,k} \;=\; N^\alpha \Big\{ \frac 1{N^\alpha} 
\;+\; \sum_{j=0}^{N-1} \frac {\tilde Z_{N-j,k-1}}
{a(j) a(N-j)} \Big\}
\end{equation*}
permits to extend it to $\kappa =k$.
\end{proof}

For any $\ell\ge 1$, let $E_{N,S}(\ell)$ be the subset of
$E_{N,S}$ of all configurations with at most $N-\ell$ particles per
site:
$$
E_{N,S}(\ell) \;=\;\big\{ \eta \in E_{N,S} : \eta_x \le
N-\ell\,,\, \forall x\in S \big\}\;.
$$
Next lemma shows that the measure $\mu_N$ is concentrated on
configurations in which all particles but a finite number accumulate
at one site.

\begin{lemma}
\label{s01}
There exists a constant $C_{\kappa}>0$ which
only depends on $\alpha$ and $\kappa$, such that for every integer
$\ell >0$,
$$
\sup_{N > \ell} \Big\{ N^{\alpha} \sum_{\eta\in E_{N,S}(\ell)} 
\frac{1}{a(\eta)}\Big\} \;\le\;
\frac{C_{\kappa}}{\ell^{\alpha -1}}\;\cdot
$$
\end{lemma}

\begin{proof}
We proceed by induction on $\kappa$. For $\kappa=2$ the statement is
easily checked. Now, suppose the claim holds for $2 \le \kappa \le
k-1$. Fix some $x$ in $S$. The left hand side of the inequality in the
statement can be written as
$$
\sum_{\eta\in E_{N,S }(\ell)} \frac{N^{\alpha}}{a(\eta_x)
  a(N-\eta_x)} \, \frac{\big(N-\eta_x\big)^{\alpha}}
{\prod_{y\not = x} a(\eta_y)}\; \cdot
$$
This sum is equal to
\begin{equation}
\label{f02}
\Big\{ \,\sum_{0\le i \le \ell/2} + \sum_{\ell/2 < i \le N-\ell}\,
\Big\} \frac{N^{\alpha}}{a(i)a(N-i)} 
\sum_{\xi\in E_{N-i,S\setminus\{x\}}(\ell -i)} \frac{(N-i)^{\alpha}}{a(\xi)}\;,
\end{equation}
where the second sum is equal to zero if $\{i : \ell/2 < i \le
N-\ell\}$ is empty. We examine the two terms of this expression
separately. By the induction assumption, the first sum is bounded
above by
\begin{equation*}
\sum_{i=0}^{\ell/2}\frac{N^{\alpha}}{a(i)a(N-i)} 
\, \frac{C_{k-1}}{(\ell - i)^{\alpha-1}}\;\cdot
\end{equation*}
By the previous lemma, this sum is less than or equal to
$$
\frac{2^{\alpha-1}C_{k-1}}{\ell^{\alpha-1}} 
\sum_{i=0}^{\ell/2} \frac{N^{\alpha}}{a(i)a(N-i)} 
\; \le \; \frac{2^{\alpha-1}C_{k-1} \tilde Z_{N,2}}{\ell^{\alpha-1}}
\; \le \; \frac{2^{\alpha-1}C_{k-1} A_2}{\ell^{\alpha-1}}\;\cdot
$$
On the other hand, by Lemma \ref{s02} and the induction assumption for
$\kappa=2$, the second term in \eqref{f02} is less than or equal to
$$
\sum_{\ell/2 < i \le N-\ell} \frac{N^{\alpha}}{a(i)a(N-i)}
\tilde Z_{N-i,k-1} \;\le\; A_{k-1} \,C_2 \, (2/\ell)^{\alpha-1} \;\cdot
$$
This concludes the proof of the lemma.
\end{proof}

For $N\ge 2$, $0\le \ell \le N$, $x\in S$, denote by $ E^{x,\ell}_{N}$
the set of configurations in $E_{N, S}$ with at least $N-\ell$
particles at site $x$:
\begin{equation*}
E^{x,\ell}_{N} \;=\; \{\eta\in E_{N,S} : \eta(x)\ge N-\ell
\}\; .
\end{equation*}
Recall the definition of the set $S_\star$.  Next lemma shows that the
$\mu_N$--measure of the set $E^{x,\ell}_{N}$ decays exponentially if
$x$ does not belong to $S_\star$.

\begin{lemma}
\label{s03}
For each $\kappa\ge 2$, there exists a finite constant $C_\kappa$,
depending only on $\kappa$ and $\alpha$, such that 
$$
N^{\alpha} \sum_{\eta\in E_{N}^{x,\ell}} 
\frac{m_\star^\eta}{a(\eta)} \; \le \; C_{\kappa} \, m_\star(x)^{N-\ell}
$$
for all $N>\ell$.
\end{lemma}

\begin{proof}
Fix $\kappa\ge 2$ and $x$ in $S$.  The expression on the left hand
side of the statement of the lemma is bounded by
\begin{equation*}
m_\star(x)^N \;+\; N^{\alpha} \sum_{i=N-\ell}^{N-1} 
\frac{m_\star(x)^{i}}{i^\alpha} 
\sum_{\xi \in E_{N-i,S\setminus\{x\}}} \frac{1}{a(\xi)}\;\cdot
\end{equation*}
By Lemma \ref{s02} and since $m_\star(x) \le 1$, the second term is
less than or equal to
\begin{equation*}
N^{\alpha} \sum_{i=N-\ell}^{N-1} 
\frac{m_\star(x)^{i}}{i^\alpha (N-i)^\alpha} \tilde Z_{N-i,\kappa-1} 
\;\le\; A_2 \, A_{\kappa-1} \, m_\star(x)^{N-\ell}\;,
\end{equation*}
which concludes the proof of the lemma.
\end{proof}

If $\ell < N/2$, the sets $\{E^{x,\ell}_{N} : x\in S\}$ are
pairwise disjoint and
$$
E_{N,S}\setminus \bigcup_{x\in S} E^{x,\ell}_{N}
\;=\; E_{N,S}(\ell+1)\;.
$$
It follows from the two previous lemmata that the sum in the
definition of $Z_{N,S}$ restricted to configurations on $
E^{x,\ell}_{N}$, $x\in S_\star$, is close to $\kappa_\star^{-1}
Z_{N,S}$ for $\ell$ (and consequently $N$) large.

\begin{lemma}
\label{esti1}
For each $\kappa\ge 2$, there exists a constant $C_{\kappa}>0$, which
only depends on $\alpha$ and $\kappa$, such that for every integer
$\ell >0$ and $x\in S_\star$,
$$
\sup_{N > 2\ell } \; \Big| N^{\alpha} \sum_{\eta\in E_{N}^{x,\ell}} 
\frac{m_\star^\eta}{a(\eta)} - \frac{Z_{N,S}}{\kappa_\star} \Big| 
\; \le \; \frac{C_{\kappa}}{\ell^{\alpha-1}}\;\cdot
$$
\end{lemma}

\begin{proof}
As we have observed, for $0 < \ell < N/2$,
\begin{equation*}
Z_{N,S} \;=\; N^{\alpha} \sum_{x\in S} \sum_{\eta\in
E_{N}^{x,\ell}} \frac{m_\star^\eta}{a(\eta)}  
\;+\; N^{\alpha} \sum_{\eta\in E_{N,S}(\ell +1)} 
\frac{m_\star^\eta}{a(\eta)}  \;\cdot
\end{equation*}
By symmetry, for $x$, $y$ in $S_\star$,
\begin{equation*}
\sum_{\eta\in E_{N}^{x,\ell}} \frac{m_\star^\eta}{a(\eta)}
\;=\; \sum_{\eta\in E_{N}^{y,\ell}}
\frac{m_\star^\eta}{a(\eta)}\; \cdot
\end{equation*}
Hence, if $x$ belongs to $S_\star$,
\begin{equation*}
Z_{N,S} \;=\; \kappa_\star N^{\alpha} \sum_{\eta\in
E_{N}^{x,\ell}} \frac{m_\star^\eta}{a(\eta)} 
\;+\; N^{\alpha} \sum_{y\not \in S_\star} \sum_{\eta\in
E_{N}^{y,\ell}} \frac{m_\star^\eta}{a(\eta)}
\;+\; N^{\alpha} \sum_{\eta\in E_{N,S}(\ell+1)} 
\frac{m_\star^\eta}{a(\eta)}  \;\cdot
\end{equation*}
The statement now follows from the two previous lemmata.
\end{proof}

We are now in a position to prove the main result of this section.

\begin{proof}[Proof of Proposition \ref{zk}]
Fix a site $x$ in $S_\star$. By the previous lemma,
\begin{equation*}
\lim_{N\to\infty} \kappa_\star^{-1} Z_{N,S} \;=\; 
\lim_{\ell \to\infty}\lim_{N\to\infty} N^{\alpha} 
\sum_{\eta\in  E_{N}^{x,\ell}} \frac{m_\star^\eta}{a(\eta)} \;\cdot
\end{equation*}
Since $x$ belongs to $S_\star$, the previous sum is equal to
\begin{equation*}
\sum_{j= 0}^\ell \frac{N^{\alpha}}{(N-j)^\alpha}
\sum_{\xi \in E_{j, S\setminus\{x\}}} \frac{m_\star^\xi}{a(\xi)}\; ,
\end{equation*}
As $N\uparrow\infty$ and $\ell\uparrow\infty$,
this expression converges to
\begin{equation*}
\sum_{j\ge 0} \sum_{\xi \in E_{j, S\setminus\{x\}}} \frac{m_\star^\xi}{a(\xi)}\; 
=\; \prod_{y\not = x} \sum_{j\ge 0} \frac{m_\star(y)^j}{a(j)}\;
=\; \prod_{y\not = x} \Gamma_y \; .
\end{equation*}
This concludes the proof of the proposition.
\end{proof}

We close this section showing that
\begin{equation}
\label{delta}
\lim_{N\to\infty} \mu_N(\Delta_N) \;=\; 0\;.
\end{equation}
Recall the definition of the set $S_\star$ and of the sets $\ms
E^x_N$, $x\in S_\star$. Since
\begin{equation*}
\Delta_N \; = \; \Big[ \bigcup_{z\in S\setminus S_\star} 
\big\{\eta : \eta_z > b_N(z) \big\}  \Big] \; \bigcup\;  
\Big[ \bigcap_{x\in S_\star}
\big\{\eta : \eta_x < N - \ell_N \big\} \Big] \;,
\end{equation*}
intersecting the second set with the partition $A= \cap_{z\in
  S\setminus S_\star} \big\{\eta : \eta_z < N - \ell_N \big\}$ and
$A^c$, we get that
\begin{equation*}
\Delta_N \; \subset \; \bigcup_{z\in S\setminus S_\star} 
E^{z,c_N}_N  \; \bigcup\; E_{N,S}(\ell_N + 1) \;,
\end{equation*}
where $c_N = \min\{ \ell_N , N-b_N(z) : z\in S\setminus
S_\star\}$. Hence, assertion \eqref{delta} follows from Lemma
\ref{s01}, assumption (\ref{f18}) and Lemma \ref{s03}.

\section{Lower bound}
\label{sec3}

In this section we prove a lower bound for the capacity. It might be
simpler in a first reading to assume that $m$ is constant so that
$S=S_\star$.

For $b,\ell\ge 3$ and $x$, $y$ in $S_\star$, $x\not =y$, consider the tube
${L}^{x,y}_{N}$ defined by
\begin{equation*}
{L}^{x,y}_{N} \;=\; \Big\{\, \eta\in E_N : \eta_x + \eta_y \ge N -
\ell \; ; \; \eta_z \le b \; , \;z\in S\setminus S_\star \,\Big\}\;.
\end{equation*}
Clearly, $L^{x,y}_N=L^{y,
x}_N$ for any $x,y\in S_\star$. We claim that
for each $x\in S_\star$ and every $N$ sufficiently large
\begin{equation}
\label{f12}
L^{x,y}_N \cap L^{x,z}_N \;\subset\; \mce{x}\;, 
\quad y,z\in S_\star\setminus\{x\}\;.
\end{equation}
Indeed, let $\eta\in L^{x,y}_N \cap L^{x,z}_N$. First, $b\le
\inf_{z\in S\setminus S_\star}b_N(z)$ for any $N$ sufficiently large
in view of \eqref{f18}. On the other hand, $\eta_z \le \ell$ because
$\eta$ belongs to $ L^{x,y}_N$. Hence, $\eta_x \ge N-2\ell$ since
$\eta$ belongs to $ L^{x,z}_N$. Since $\ell_N\to \infty$, this shows
that $\eta_x\ge N - \ell_N$, for $N$ large enough and we conclude that
$\eta\in \mce x$. Moreover, it follows from this argument that, for
$N$ sufficiently large,
\begin{equation}
\label{f13}
 L^{x,y}_N \cap L_N^{z,w}\not = \varnothing \quad 
\textrm{if and only if} \quad\{x,y\}\cap\{z,w\}\not =\varnothing \;. 
\end{equation}

\begin{proposition}
\label{s05} 
Assume that $\kappa_\star\ge 2$. Fix a nonempty subset $S_\star^1
\subsetneq S_\star$ and denote $S_\star^2= S_\star\setminus
S_\star^1$. Then,
\begin{equation*}
\liminf_{N\to\infty} N^{1+\alpha} \Cap_N\big(\ms E_N(S_\star^1), \ms
E_N({S_\star^2})\big) \;\ge\; \frac 1{M_\star \, \kappa_\star \,
  \Gamma(\alpha) \, I_\alpha} 
\sum_{x \in S_\star^1, y\in S_\star^2} \Cap_S(x,y) \; .
\end{equation*}
\end{proposition}

\begin{proof}
Fix a function $F$ in $\mc C_N(\ms E_N(S_\star^1), \ms
E_N(S_\star^2))$. By definition,
\begin{equation*}
D_N(F) \;=\; \frac 12 \sum_{z,w \in S} \sum_{\eta\in E_N} \mu_N(\eta)
\, r(z,w) \, g(\eta_z) \, \{ F(\sigma^{zw}\eta) - F(\eta)\}^2 \;.
\end{equation*}
We may bound from below the Dirichlet form $D_N(F)$ by
\begin{equation*}
\frac 12  \sum_{x\in S_\star^1} \sum_{y \in S_\star^2} \sum_{z,w \in S}  
\sum_{\eta\in L^{x,y}_N} \mu_N(\eta)
\, r(z,w) \, g(\eta_z) \, \{ F(\sigma^{zw}\eta) - F(\eta)\}^2\;.
\end{equation*}
In this inequality, we are neglecting several terms corresponding to
configurations $\eta$ which do not belong to $\cup_{x \in S_\star^1,
  y\in S_\star^2} L^{x,y}_N$. On the other hand, some configurations
are counted more than once because the sets $\{ L^{x,y}_N : x\in
S_\star^1 ,y \in S_\star^2\}$ are not disjoints. However, by
\eqref{f13}, if $ L^{x,y}_N$ and $ L^{x',y'}_N$ are different strips
and $\eta$ belongs to $ L^{x,y}_N \cap L^{x',y'}_N$ then $x= x'$ and
$y\not = y'$ (recall that $ L^{x,y}_N= L^{y,x}_N$). In consequence,
$\eta_x \ge N-2\ell$. In particular, for $N$ large enough, $\eta$ and
$\sigma^{zw}\eta$ belong to $\ms E^x_N$ for all $z,w\in S$, so that
$F(\sigma^{zw}\eta)= F(\eta)$ because $F$ is constant on $\ms E^x_N$.

The proof of the lower bound has two steps. We first use the
underlying random walk to estimate the Dirichlet form $D_N(F)$ by the
capacity of this random walk multiplied by the Dirichlet form of a zero
range process on two sites. This remaining Dirichlet form is easily
bounded by explicit computations.

Fix $x\in S_\star^1$, $y \in S_\star^2$. Denote by $\mf d_x$, $x\in
S$, the configuration with one and only one particle at $x$, and
agree that summation of configurations is performed
componentwise. The change of variables $\xi = \eta - \mf d_z$ shows
that
\begin{equation*}
\begin{split}
& \frac 12 \, \sum_{z,w \in S} \sum_{\eta\in
  L^{x,y}_N} \mu_N(\eta)
\, r(z,w) \, g(\eta_z) \, \{ F(\sigma^{zw}\eta) - F(\eta)\}^2 \\
& \qquad \;=\; \frac 12 \, \sum_{z,w \in S} \sum_{\substack{\xi \in E_{N-1} \\ 
\xi + \mf d_z \in  L^{x,y}_N}} \frac{N^\alpha}{Z_{N,S}} \, \frac
{m_\star^\xi}{a(\xi)}\, m_\star (z)
\, r(z,w) \, \{ F(\xi + \mf d_w) - F(\xi + \mf
d_z)\}^2 \;.
\end{split}
\end{equation*}
This sum is clearly bounded below by
\begin{equation*}
\frac 12 \, \sum_{z,w \in S} \sum_{\substack{\xi \in E_{N-1} \\ 
\xi_x + \xi_y \ge N - \ell \\ \xi_z \le b-1, \forall z\in S\setminus
S_\star}} \frac{N^\alpha}{Z_{N,S}} \, 
\frac {m_\star^\xi} {a(\xi)} \, m_\star (z)
\, r(z,w) \, \{ F(\xi + \mf d_w) - F(\xi + \mf d_z)\}^2 \;.
\end{equation*}
Fix a configuration $\xi$ in $E_{N-1}$ and consider the function
$f:S\to \bb R$ given by $f(v) = \{F(\xi + \mf d_v) - F(\xi + \mf
d_{y})\}/\{F(\xi + \mf d_{x}) - F(\xi + \mf d_{y})\}$. Note that
$f(x)=1$, $f(y)=0$. Moreover, if we recall the expression
\eqref{f14} of the Dirichlet form of the underlying random walk,
\begin{equation*}
\begin{split}
& \frac 12 \, \sum_{z,w \in S} m_\star (z)
\, r(z,w) \, \{ F(\xi + \mf d_w) - F(\xi + \mf d_z)\}^2 \\
&\qquad\qquad \;=\; \frac 1{M_\star}\, D_S(f) \, \{ F(\xi + \mf d_{x}) - 
F(\xi + \mf d_{y})\}^2\;.   
\end{split}
\end{equation*}
Since $f(x)=1$, $f(y)=0$, the previous expression is bounded below
by
\begin{equation*}
\frac 1{M_\star}\, \Cap_S(x,y) \, \{ F(\xi + \mf d_{x}) 
- F(\xi + \mf d_{y})\}^2 \;.
\end{equation*}

Up to this point we proved that the Dirichlet form of $F$ is bounded
below by 
\begin{equation*}
\frac 1{M_\star}\, \sum_{x \in S_\star^1, y\in S_\star^2}
\Cap_S(x,y) \, \sum_{\substack{\xi \in E_{N-1} \\ 
\xi_x + \xi_y \ge N - \ell \\ \xi_z \le b-1, \forall z\in S\setminus
S_\star}} \frac{N^\alpha}{Z_{N,S}} \, 
\frac {m_\star^\xi} {a(\xi)}  \, \{ F(\xi + \mf d_{x}) 
- F(\xi + \mf d_{y})\}^2 \;.
\end{equation*}

Fix $x_0\in S^1_\star$, $y_0\in S_\star^2$ and let $S_0 := S \setminus
\{x_0,y_0\}$. For each $k\ge 0$, let $B_k = B^{x_0,y_0}_k$ be the set
of configurations on $S_0$ given by
$$
B_k \;=\; \Big\{ \, \zeta \in \bb N^{S_0} : \sum_{v\in S_0}
\zeta_v = k \;;\; \zeta_z\le b-1 \,,\, z\in S\setminus S_\star  \,\Big\}\;.
$$
For $\zeta$ in $B_k$, let $G_\zeta : \{0, \dots, N -1 - k \} \to \bb
R$ be defined as $G_\zeta(i) = F(\xi)$, where $\xi\in E_{N-1}$ is the
configuration given by $\xi_v = \zeta_v$, $v\in S_0$, $\xi_{x_0} = i$
and $\xi_{y_0}=N-1-k-i$. With this notation, for $x_0 \in S^1_\star$,
$y_0\in S^2_\star$ fixed, we may rewrite the second sum in the
previous formula as
\begin{equation*}
\frac{N^\alpha}{Z_{N,S}} \sum_{k=0}^{\ell} 
\sum_{\zeta \in B_k  }  \, 
\frac {m_\star^\zeta} {a(\zeta)}  \sum_{i=0}^{N-2-k} 
\frac {1} {a(i) \, a(N-1-k-i)}
\, \{ G_\zeta(i+1)  - G_\zeta(i)\}^2
\end{equation*}
because $m_\star (x_0) = m_\star(y_0)=1$. Note that $G_\zeta$ is equal
to $0$ on the set $\{0, \dots, \ell_N - k \}$, and equal to $1$ on the
set $\{N-\ell_N, \dots, N -1-k\}$. We may therefore restrict the sum
over $i$ to a subset. It is easy to derive a lower bound for
\begin{equation*}
\sum_{i=\ell_N - k}^{N-\ell_N-1} \frac {1} 
{a(i) \, a(N-1-k-i)} \, \{ G_\zeta(i+1)  - G_\zeta(i)\}^2\; .
\end{equation*}
The function $G$ which minimizes this expression is given by $G(N -
\ell_N) =1$,
\begin{equation*}
G(i+1)  - G(i) \; =\; \frac{1}{K_N}\, a(i) \, a(N-1-k-i)\;,
\quad i\in \big[ \ell_N-k \,,\, N-\ell_N-1 \big]\;, 
\end{equation*}
where $K_N$ is a normalizing constant to ensure the boundary
condition $G(\ell_N-k)=0$. The respective lower bound is
\begin{equation*}
\Xi_N(x_0,y_0) \;:=\; \Big\{ \sum_{i=\ell_N - k}^{ N-\ell_N-1}
a(i) \, a(N-1-k-i) \Big\}^{-1}\;.
\end{equation*}
This expression depends on the configuration $\zeta$ only through its
number of particles. Moreover, for every fixed $k$, $N^{1+2\alpha}
\Xi_N(x_0,y_0)$ converges to $I_\alpha^{-1}$ as $N\uparrow\infty$.

In conclusion,
\begin{equation*}
N^{\alpha+1}\, D_N(F) \;\ge\; \frac 1{M_\star} \sum_{x \in S^1_\star, y\in S^2_\star} 
\Cap_S(x,y) \, \frac{N^{2\alpha+1}}{Z_{N,S}} \sum_{k=0}^{\ell} \sum_{\zeta
\in B_k} \Xi_N(x,y) \, \frac {m_\star^\zeta} {a(\zeta)}  \;\cdot
\end{equation*}
By Proposition \ref{zk} and the above conclusions, as $N\uparrow
\infty$, the right hand side converges to
$$
\frac 1{M_\star\,I_\alpha \,Z_{S}} \sum_{x \in S^1_\star, y\in S^2_\star} 
\Cap_S(x,y) \, \sum_{k=0}^{\ell} \sum_{\zeta \in B_k} 
\frac {m_\star^\zeta} {a(\zeta)}  \;\cdot
$$
Recall that $\ell$ and $b$ are free parameters introduced in the
definition of the strip $ L^{x,y}_N$. Thus, letting $b\uparrow \infty$
and then $\ell \uparrow \infty$, the second sum in the last expression
converges to
\begin{equation*}
\sum_{k\ge 0} \, \sum_{\zeta \in E_{k,S\setminus\{x,y\}}} \, \frac {m_\star^\zeta}
{a(\zeta)} \;=\;
\prod_{z\in S\setminus \{x,y\}} \,\sum_{j\ge 0} \frac {m_\star (z)^j}
{a(j)} \;=\;  
\prod_{z\in S\setminus \{x,y\}} \Gamma_z \; = \; \frac{Z_S }
{\kappa_\star \Gamma(\alpha)} \;\cdot
\end{equation*}
For the last equation we have used the explicit formula of $Z_S$
presented just before Proposition \ref{zk}. This proves the lemma.
\end{proof}

\section{Upper bound}
\label{sec4}

We prove in this section an upper bound for the capacity. As in the
previous section, it might be simpler in a first reading to assume
that $m$ is constant so that $S=S_\star$.

\begin{proposition}
\label{s08}
Assume that $\kappa_\star\ge 2$. Fix a nonempty subset $S_\star^1
\subsetneq S_\star$ and denote $S_\star^2= S_\star\setminus
S_\star^1$. Then,
\begin{equation*}
\limsup_{N\to\infty} N^{1+\alpha} \Cap_N\big(\ms E_N(S_\star^1), \ms
E_N({S_\star^2})\big) \;\le\; \frac 1{M_\star \, \kappa_\star \,
  \Gamma(\alpha) \, I_\alpha} 
\sum_{x \in S_\star^1, y\in S_\star^2} \Cap_S(x,y) \; .
\end{equation*}
\end{proposition}

In view of the variational formula for the capacity, to obtain an
upper bound for $\Cap_N(\ms E_N(S_\star^1), \ms E_N(S_\star^2) )$, we
need to choose a suitable function belonging to $\mc C_N(\ms
E_N(S_\star^1), \ms E_N(S_\star^2))$ and to compute its Dirichlet
form. Recalling the proof of the lower bound, we expect this candidate
to depend on the function which solves the variational problem for the
capacity of the underlying random walk and on the optimal function for
the zero range process with two sites.

To introduce the candidate, fix $x\in S^1_\star$, $y\in S^2_\star$ and
recall the definition of the tube $L^{x,y}_N$. In view of the proof of
the lower bound, the optimal function $F\in \mc C_N(\ms
E_N(S_\star^1), \ms E_N(S_\star^2))$ on the tube $L^{x,y}_N$ should
satisfy
\begin{equation*}
\begin{split}
F(\xi + \mf d_w) - F(\xi + \mf d_z) & \;=\; \{ {\bf f_{xy}}(w) 
- {\bf f_{xy}} (z) \} \, \{ F(\xi + \mf d_{x}) - F(\xi + \mf d_{y}) \} \, 
\\ & \;=\; \{ {\bf f_{xy}} (w) - {\bf f_{xy}} (z) \} \,
\{ G(\xi_x + 1) - G(\xi_x) \}\;,
\end{split}
\end{equation*} 
where $\bf f_{x,y}$ is the function which solves the variational
problem \eqref{ff1} in $\mc B(x,y)$ for the capacity of the underlying
random walk, and $G$ is the function appearing in the proof of the
lower bound.

Since, on the tube $L^{x,y}_N$, $\sum_{z\not = x,y} \xi_z \le
\ell_N$ and $G$ is a smooth function, paying a small cost we may
replace $\xi_x$ in the previous formula by $\xi_x + \sum_{z\in A}
\xi_z$ for any suitable set $A\subset S\setminus \{x,y\}$. The natural
candidate on the strip $L^{x,y}_N$ is therefore
\begin{equation*}
\hat F_{xy} (\xi) \;:=\; \sum_{j=1}^{\kappa-1} 
\{ {\bf f_{xy}} (z_j) - {\bf f_{xy}} (z_{j+1}) \}\, G \big( \xi_{z_1} +
\cdots + \xi_{z_j}\big) \;,
\end{equation*}
where $x=z_1, z_2, \dots, z_\kappa =y$ is an enumeration of $S$ such
that ${\bf f_{xy}} (z_j) \ge {\bf f_{x,y}} (z_{j+1})$ for $1\le j<
\kappa$. A simple computation shows that this function has the
required properties listed in the previous paragraph. 

Since the tubes $L^{x,y}_N$, $x\in S^1_\star$, $y\in S^2_\star$,
are essentially disjoints, the candidate $F$ should be equal to $\hat
F_{xy}$ on each tube $L^{x,y}_N$ and equal to some appropriate convex
combination of these functions on the complement.

We hope that this informal explanation helps to understand the
rigorous and detailed definition of the candidate we now present. Let
$\ms D\subset \bb R^{S}$ be the compact subset
$$
\ms D \;:=\; \{ u\in \bb R_{+}^{S} : \sum_{x\in S} u_x = 1 \}\;.
$$
For each different sites $x,y\in S$ and $\delta>0$, consider the
subsets of $\ms D$
$$
\ms D^x_\delta \;:=\;\{ u\in \ms D : u_x > 1-\delta \} 
\quad \text{and} \quad 
\ms L^{xy}_{\delta} \;:=\; \{ u\in \ms D : u_x + u_y\ge 1-\delta \}
$$
Clearly $\ms L^{xy}_{\delta} = \ms L^{yx}_{\delta}$ for any $x,y\in S$.

Fix an arbitrary $0<\epsilon<1/6$ and $x$ in $S$. Let $\ms K^x_y = \ms
K^x_y (\epsilon) := \ms L^{xy}_{\epsilon}\setminus \ms
D^x_{3\epsilon}$, $y\not = x$. Since $\ms K^x_y$, $y\in
S\setminus\{x\}$, is a collection of pairwise disjoint compact subsets
of $\ms D$, there is a family of smooth functions
$$
\Theta^x_{y}: \ms D\to [0,1]\;, \quad y\in S\setminus\{x\} \;,
$$
such that $\sum_{y \in S\setminus \{x\}} \Theta^x_{y}(u) = 1$ for all
$u$ in $\ms D$, and $\Theta^x_{y}(u) = 1$ for all $u$ in $\ms
K^x_y$ and $y\in S \setminus \{x\}$. 

Clearly, the sets $\ms L^{xy}_{\epsilon}$ are macroscopic versions of
the strips $L^{x,y}_N$. The functions $\Theta^x_{y}$ will be used to
define the candidate function in the complement of the cylinders
$L^{x,y}_N$. 

Let $H:[0,1]\to [0,1]$ be the smooth function given by
$$
H(t)\;:=\; \frac{1}{I_{\alpha}} \,\int_0^{\phi(t)} u^{\alpha}(1-u)^{\alpha}\,du\;,
$$
where $I_{\alpha}$ is the constant defined in (\ref{defi}) and
$\phi:[0,1]\to [0,1]$ is a smooth bijective function such that
$\phi(t)+\phi(1-t)=1$ for every $t\in [0,1]$ and $\phi(s)=0$ $\forall
s\in[0, 3\epsilon]$. It can be easily checked that
\begin{equation}\label{ph1}
H(t)+H(1-t)=1 \;,\quad \forall t\in [0,1]\;,
\end{equation}
$H|_{[0,3\epsilon]} \equiv 0$ and $H|_{[1-3\epsilon,1]}\equiv 1$. The
function $H$ is a smooth approximation of the function $G$ which
appeared in the proof of the lower bound.

Recall that $x\in S$ is fixed. For each $y\in S\setminus \{x\}$
consider the function ${\bf f_{xy}}:S\to [0,1]$ in $\mc B(x,y)$ such
that
\begin{equation*}
D_S({\bf f_{xy}}) \;=\; \Cap_S(x,y) \;=\; \inf_{f\in \mc B(x,y)} D_S(f)\;.
\end{equation*}
It is well known that ${\bf f_{xy}}(z)$ is equal to the probability
that the random walk with generator $\mc L_S$ reaches $x$ before $y$
when it starts from $z$. 

For each $y\in S\setminus\{x\}$ fix an enumeration
\begin{equation}\label{enum}
x\,=\,z_1\,, \;\;z_2\;,\;\;\dots\;,\;\;z_{\kappa}\,=\,y
\end{equation}
of $S$ satisfying ${\bf f_{xy}}(z_j)\ge {\bf f_{xy}}(z_{j+1})$ for
$1\le j\le \kappa -1$ and define $F_{xy}:E_N\to \bb R$ as the convex
linear combination
$$
F_{xy}(\eta) \;:=\; \sum_{j=1}^{\kappa -1}
 \{ {\bf f_{xy}}(z_j) - {\bf f_{xy}}(z_{j+1}) \} \, 
F^j_{xy}(\eta)\;,\quad \eta\in E_N\;,
$$
where each $F^j_{xy}:E_N\to \bb R$, $1\le j \le \kappa -1$, is given by
$F^1_{xy}(\eta)=H(\eta_x/N)$ and
\begin{equation}
\label{dfj}
F^j_{xy}(\eta) \;:=\; H\Big(\, \frac{\eta_x}{N} + 
\min\big\{\, \frac 1N \sum_{i=2}^{j}\eta_{z_i} \,; 
\epsilon  \, \big\} \,\Big) \;,\quad \eta\in E_N \;,
\end{equation}
for $2\le j \le \kappa-1$. 

The function $F_{xy}$ just defined is a smooth approximation of the
function $\hat F_{x,y}$ defined at the beginning of this section. It
is therefore the candidate to solve the variational problem for the
capacity on the tube $L^{x,y}_N$. It remains to define $F_{xy}$
in the exterior of the cylinders.

Let $F_x:E_N\to \bb R$ be given by
$$
F_x(\eta) \;:=\; \sum_{y\in S\setminus \{x\}} \Theta^x_y(\eta/N) \, 
F_{xy}(\eta)\;,
$$
where each $\eta/N$ is thought of as a point in $\ms D$ and $\{
\Theta^x_y : y \in S \setminus \{x\}\}$ is the partition of unity
established before. 

The following properties of $F_x$ are helpful in the proof of
Proposition \ref{s08}. It is easy to check that
\begin{equation}
\label{pf0}
F_x(\eta)=F_{xy}(\eta) \quad \textrm{for $\;\eta/N \in 
\ms L^{xy}_{\epsilon}$}\;.
\end{equation}
Indeed, if $\eta/N$ belongs to $\ms K^x_y$, $\Theta^x_y (\eta/N)=1$
proving the identity claimed. On the other hand, if $\eta/N$ belongs
to $\ms D^x_{3\epsilon}$, by definition of $H$, $F^j_{xz} (\eta) =1$
for all $z\in S\setminus \{x\}$, $1\le j\le \kappa-1$, so that
$F_{xz}(\eta) = F_x(\eta)$. By similar reasons,
\begin{multline}
\label{pf2}
F_x\equiv 1\quad \text{on $\{\eta\in E_N : \eta_x\ge
  (1-3\epsilon)N\}$} \\ 
\text{and} \quad F_x\equiv 0 \quad 
\text{on $\{\eta\in E_N : \eta_x\le 2\epsilon N$\}}\;.
\end{multline}
The minimum in definition (\ref{dfj}) is introduced precisely to
fulfill the second assertion in (\ref{pf2}). In particular, if $\eta/N
\in \ms D^z_{2\epsilon}$ for some $z\in S$ then
\begin{equation}
\label{pf3}
F_x(\eta) \;=\; {\bf 1}\{z=x\}\;.
\end{equation}

Since $H$, as well as each $\Theta^x_y$, is a smooth function, there
exists a finite constant $C_\epsilon$, which depends on $\epsilon$
through the definition of the smooth functions, but does not depend on
$N\ge 1$, such that
\begin{equation}
\label{pf1}
\max_{\eta\in E_N} | F_x(\sigma^{zw}\eta) - F_x(\eta) | \;\le\; 
\frac{C_{\epsilon}}{N}
\end{equation}
for every $z,w\in S$. 

Let  
$$
\ms I^{xy}_N \;:=\; \big\{ \eta\in E_N : \eta_x 
+ \eta_y \ge N - \ell_N \big\}\;, \quad x\not=y \in S \;.
$$
Clearly, $\ms I^{xy}_N=\ms I^{yx}_N$, $x$, $y\in S$ and, for every $N$
large enough, $\ms I^{xy}_N\subseteq \ms L^{xy}_\epsilon$. Let $\ms
I^x_N :=\cup_{y\in S\setminus \{x\}} \ms I^{xy}_N$. In what follows,
the value of the constant $C_{\epsilon}$ may change from line to line,
but will never depend on $N$.

\begin{lemma}
\label{est1}
For each $x\in S$ and every $N\ge 1$ large enough,
$$
\frac 12 \sum_{\eta\in E_N\setminus \ms I^x_N} \, 
\sum_{z,w\in S} \mu_N(\eta)g(\eta_z)r(z,w)
\big\{ F_x(\sigma^{zw} \eta) - F_x(\eta) \big\}^2 
\;\le \;  \frac{C_\epsilon \, m_\star(x)^{\epsilon N}}{N^{\alpha+1} 
\,(\epsilon \,\ell_N)^{\alpha -1}}\;\cdot
$$
\end{lemma}

\begin{proof}
By property (\ref{pf2}), we can restrict the sum in the left hand side
to configurations $\eta\in E_N\setminus \ms I^x_N$ satisfying
$\epsilon N \le \eta_x \le (1-\epsilon)N$. So, by (\ref{pf1}), the
left hand side of the above inequality is bounded above by
$$
\frac{C_\epsilon}{N^2} \,\sum_{\substack{\eta\in E_N\setminus \ms
    I^x_N \\ \epsilon N \le \eta_x \le (1-\epsilon )N}} \mu_N(\eta)\;.
$$
This expression is bounded above by
$$
\frac{N^\alpha C_\epsilon}{Z_{N,S} N^2}\, 
\sum_{\epsilon N \le i \le (1-\epsilon)N} \; 
\sum_{\substack{\eta : \eta_x=i \\ \max\{\eta_y : y \not = x\}\le N-i-
    \ell_N}} \frac{m_\star^\eta }{a(\eta)}\; ,
$$
which can be re-written as
$$
\frac{N^\alpha C_\epsilon}{Z_{N,S} N^2}\, 
\sum_{\epsilon N \le i \le (1-\epsilon)N} 
\frac{m_\star(x)^i}{a(i) a(N-i)} \, 
\Big\{ \,(N-i)^{\alpha} \!\!\!
\sum_{\zeta \in E_{N-i, S\setminus\{x\}}( \ell_N )} 
\frac{m_\star^\zeta }{a(\zeta)} \, \Big\} \; \cdot
$$
By Lemma \ref{s01} for the expression inside braces, last expression
is bounded above by
$$
\frac{ C_{\epsilon}\, m_\star(x)^{\epsilon N} }
{Z_{N,S} \,N^2 \,\ell_N^{\alpha -1} } \, 
\Big\{\, N^\alpha \sum_{\epsilon N \le i \le (1-\epsilon)N} 
\frac{1}{a(i) a(N-i)} \,\Big\}  \; \cdot
$$
By Lemma \ref{s01} once more and Proposition \ref{zk} we obtain the
desired result.
\end{proof}

Fix a nonempty subset $S^1\subsetneq S$ and denote $S^2:= S\setminus
S^1 \not = \varnothing$. We define the function $F_{S^1}:E_N\to \bb R$
as
$$
F_{S^1}(\eta) \;:=\; \sum_{x\in S^1} F_x(\eta) \;.
$$
Let us define the following subsets of $E_N$
$$
D^x_N\;:=\; \{\eta\in E_N : \eta_x \ge N - 3\ell_N \}\;,\quad x\in S\;,
$$
so that $\ms E^x_N \subset D^x_N$, $x\in S_\star$. It follows from
(\ref{pf2}) that if $\eta\in D^x_N$ for some $x\in S$ then
\begin{equation}
\label{pfs0}
F_{S^1}(\eta) \;=\; {\bf 1}\{x\in S^1\}\;=\; F_{S^1}(\sigma^{zw}\eta)\;,
\end{equation}
for every $z,w\in S$ and every $N$ large enough. In particular,
$$
F_{S^1} \;\in\; \mc C_N\Big( \,  \bigcup_{x\in S^1}D^x_N , 
\bigcup_{y\in S^2} D^y_N \, \Big) \;.
$$
We shall use $F_{S^1}$ to get an upper bound for $\Cap_N\big( \,
\cup_{x\in S^1}D^x_N , \cup_{y\in S^2} D^y_N \, \big)$.
 
We first claim that for any $N$ large enough,
\begin{equation}
\label{pfs1}
F_{S^1}(\sigma^{zw}\eta) \;= \; 1 \;=\; F_{S^1}(\eta)  
\quad \textrm{for all } \eta \in \bigcup_{x,y\in S^1} 
\ms I^{xy}_N \text{ and } z,w\in S \;.
\end{equation}
To prove this claim, fix $x \not = y$ in $S^1$. By \eqref{pf2},
\eqref{pf0}, for $\eta/N\in \ms L^{xy}_{\epsilon}$, 
\begin{equation}
\label{ff2}
F_{S^1}(\eta) \;=\; F_{xy}(\eta) \;+\;  F_{yx}(\eta)\;.
\end{equation}

Recall from (\ref{enum}) the enumeration of $S$ defined according to
the values of ${\bf f_{xy}}$. Let $z_1,\dots,z_\kappa$ and
$w_1,\dots,w_\kappa$ be such enumerations obtained from $\bf f_{xy}$
and $\bf f_{yx}$, respectively. Since ${\bf f_{xy}}+ {\bf f_{yx}}
\equiv 1$, we can choose the enumerations in such a way that
$z_{n+1}=w_{\kappa -n}$, $0\le n\le \kappa -1$. With this convention,
an elementary computation shows that
\begin{equation*}
F_{xy}(\eta) + F_{yx}(\eta) \;=\; \sum_{j=1}^{\kappa -1} 
\{ {\bf f_{xy}}(z_j) - {\bf f_{xy}}(z_{j+1}) \} 
\big(\, F^j_{xy}(\eta) + F^{\kappa -j}_{yx}(\eta) \,\big)\;.
\end{equation*}
By (\ref{ph1}), the previous expression is equal to
\begin{equation*}
\sum_{j=1}^{\kappa -1} \{ {\bf f_{xy}}(z_j) - {\bf f_{xy}}(z_{j+1}) \} 
\;=\; 1 \;.
\end{equation*}
Claim \eqref{pfs1} follows from this identity and \eqref{ff2} since
$\ms I^{xy}_N \subset \ms L^{xy}_{\epsilon}$ for $N$ sufficiently
large.

For each subset $A \subseteq E_N$ and function $F:E_N\to \bb R$, let
$$
D_N( F ; A) \;:=\; \frac 12 \sum_{\eta\in A} \, 
\sum_{z,w\in S} \mu_N(\eta)g(\eta_z)r(z,w)
\big\{ F(\sigma^{zw} \eta) - F(\eta) \big\}^2 \;.
$$
With this notation, Lemma \ref{est1} can be stated as
\begin{equation*}
D_N(F_x ; E_N\setminus \ms I^x_N) \;\le\;   
\frac{C_{\epsilon} \, m_\star(x)^{\epsilon N}}{N^{\alpha+1}\, 
(\epsilon \,\ell_N)^{\alpha -1}}\;\quad \forall x\in S\;.
\end{equation*}
By Cauchy-Schwarz inequality,
\begin{eqnarray*}
D_N(F_{S^1} ; E_N\setminus \cup_{z\in S^1} \ms I^z_{N}) 
&\le& |S^1| \sum_{x\in S^1} D_N(F_{x} ; 
E_N\setminus \cup_{z\in S^1} \ms I^z_{N}) \\
&\le& |S^1| \sum_{x\in S^1} D_N(F_{x} ; E_N\setminus \ms I^x_{N})\;.
\end{eqnarray*}
Therefore, since $\ell_N\uparrow\infty$, it follows from Lemma
\ref{est1} that
\begin{equation}
\label{ef1}
\lim_{N\to\infty}  N^{\alpha + 1} \, D_N(F_{S^1} ; 
E_N\setminus \cup_{z\in S^1} \ms I^z_{N}) \;=\; 0\;.
\end{equation}

It remains to estimate $D_N(F_{S^1} ; \cup_{z\in S^1} \ms I^z_{N})$.
By definition of $\ms I^z_{N}$, $z\in S^1$, and by (\ref{pfs1}),
\begin{equation*}
D_N \Big(F_{S^1} ; \bigcup_{z\in S^1} \ms I^z_{N} \Big) \;=\; 
D_N\Big(\,  F_{S^1} ; { \bigcup_{\substack{x\in S^1 \\ y\in S^2 }}} 
\ms I^{xy}_N \, \Big) \; =\; \sum_{x\in S^1} \sum_{y\in S^2} 
D_N\big( F_{S^1} ; \ms I^{xy}_N \big) \;.
\end{equation*}
The last identity follows from (\ref{pfs0}) and the relation
$$
\ms I^{x_1y_1}_N \cap \ms I^{x_2y_2}_N \;\subseteq\; 
\bigcup_{z\in S} D^z_N \quad 
\textrm{for all $x_1,x_2\in S^1$ and $y_1,y_2\in S^2$}\;.
$$
Therefore, by (\ref{pf2}) and (\ref{pf0}) we finally conclude
that
\begin{equation}
\label{est2}
D_N \Big (F_{S^1} ; \bigcup_{z\in S^1} \ms I^z_{N} \Big) \;=\; 
\sum_{x\in S^1} \sum_{y\in S^2} D_N\big( F_{xy} ; \ms I^{xy}_N \big).
\end{equation}
We now provide an estimate for each term in this sum. To derive this
bound, in addition to the properties already imposed to the function
$\phi$, we also require that
\begin{equation}
\label{ad1}
\sup \big\{ \, \phi'(u) : u\in [0,1] \,\big\} \;\le\; 1+\sqrt\epsilon
\end{equation}
and
\begin{equation}
\label{ad2}
\sup\big\{\, \frac{\phi(u)}{u-\epsilon} : 
u\in [2\epsilon ,1] \,\big\} \;\le\; 1+\sqrt\epsilon\;.
\end{equation}
The first requirement can easily be accomplished since
$(1+\sqrt\epsilon)$ times the length of the interval $[3\epsilon ,
1-3\epsilon]$ is strictly greater than $1$ for $\epsilon$ small
enough. For (\ref{ad2}), it suffices that $\phi(u)\le
(u-\epsilon)(1+\sqrt\epsilon)$ for all $u\in [3\epsilon, 1]$ because
$\phi$ vanishes on $[0,3\epsilon]$. Since $(u-\epsilon)
(1+\sqrt\epsilon)>1$ for $u = 1-3\epsilon$ and every $\epsilon$ small
enough, it is possible to define a smooth function $\phi$ satisfying
(\ref{ad2}) without violating the other previously imposed properties.

According to the above discussion, in what follows we suppose that
$\epsilon$ is an arbitrary number in $(0,\epsilon_0]$ for a suitably
chosen $\epsilon_0>0$ and that $\phi$ satisfies the additional
properties (\ref{ad1}) and (\ref{ad2}).

\begin{proposition}
\label{est5}
For any $x,y\in S$, $x\not=y$,
$$
\limsup_{N\to\infty} N^{\alpha +1} D_N( F_{xy} ; \ms I^{xy}_N ) 
\;\le\;  \frac{ (1+ \sqrt{\epsilon})^{2\alpha +1} }
{ M_{\star}\kappa_\star I_{\alpha} \Gamma(\alpha) } \,
\Cap_S(x,y) \,{\bf 1}\{x,y\in S_{\star}\}\;. 
$$
\end{proposition}

\begin{proof}
Let $x=z_1, z_2, \dots, z_{\kappa}=y$ be the enumeration established
in the definition of $F_{xy}$, so that ${\bf f_{xy}}(z_n)\ge {\bf
  f_{xy}}(z_{n+1})$, $1\le n \le \kappa-1$. Fix two different
sites $z_{i}\not = z_j$ in $S$ with $1\le i<j \le \kappa$. By
definition of $F_{xy}$,
$$
F_{xy}(\sigma^{z_iz_j}\eta) - F_{xy}(\eta) \;=\; 
\sum_{n=i}^{j-1}\big( {\bf f_{xy}}(z_n)- {\bf f_{xy}}(z_{n+1}) \big)
\{ F^n_{xy}(\sigma^{z_iz_j}\eta) -  F^n_{xy}(\eta) \}\;.
$$
Thus, by the Cauchy-Schwarz inequality, the sum
\begin{equation}
\label{equ1}
\sum_{\eta\in \ms I^{xy}_N} \mu_N(\eta)g(\eta_{z_i})r(z_i,z_j)
\big\{ F_{xy}(\sigma^{z_iz_j}\eta) - F_{xy}(\eta) \big\}^2
\end{equation}
is bounded above by $\{{\bf f_{xy}}(z_i) - {\bf f_{xy}}(z_j)\}$ times
$$
\sum_{n=i}^{j-1} \big( {\bf f_{xy}}(z_{n}) - {\bf f_{xy}}(z_{n+1})
\big) \sum_{\eta\in \ms I^{xy}_N} \mu_N(\eta) g(\eta_{z_i}) r(z_i,z_j)
\big\{ F_{xy}^n(\sigma^{z_iz_j} \eta) - F_{xy}^n(\eta) \big\}^2\;.
$$
Performing the change of variables $\xi=\eta - \mf d_{z_i}$, the
second sum above is less than
\begin{equation*}
m_{\star}(z_i)r(z_i,z_j) \frac{N^{\alpha}}{Z_{N,S}} 
\sum_{\xi\in A^{xy}_N} \frac{m_{\star}^{\xi}}{a(\xi)} 
\big\{ F^n_{xy}(\xi) - F^n_{xy}(\xi + {\mf d}_{z_i}) \big\}^2 \;,
\end{equation*}
where $A^{xy}_N := \{ \xi\in E_{N-1, S} : \xi_x + \xi_y \ge N -
2\ell_N \}$. So far, we have shown that (\ref{equ1}) is bounded above
by $m_{\star}(z_i)r(z_i,z_j) N^{\alpha} Z_{N,S}^{-1} \{{\bf
  f_{xy}}(z_i) - {\bf f_{xy}}(z_j)\}$ times
$$
\sum_{n=i}^{j-1} \big( {\bf f_{xy}}(z_{n}) - {\bf f_{xy}}(z_{n+1})
\big) \sum_{\xi\in A^{xy}_N} \frac{m_{\star}^{\xi}}{a(\xi)} 
\Big\{ H \Big( \, \frac 1N + \sum_{r=1}^n \frac{\xi_{z_r}}{N} \Big) 
- H \Big( \sum_{r=1}^n \frac{\xi_{z_r}}{N} \Big) \, \Big\}^2 \;.
$$
Fix some $i\le n < j$. The second sum in the above expression may be
re-written as
\begin{multline}
\label{equ3}
\sum_{m=0}^{2\ell_N}\;\; \sum_{\zeta \in E_{m,S\setminus\{x,y\}}} 
\frac{m_\star^{\zeta}}{a(\zeta)} \sum_{\epsilon N \le k \le
  (1-2\epsilon)N} \frac{ m_\star(x)^k m_{\star}(y)^{N-m-k} }
{ a(k)a(N-m-k)} \\\Big\{ H \big( d_{k+1}(\zeta) \big) - H 
\big( d_{k}(\zeta) \big) \, \Big\}^2\;,
\end{multline}
where 
$$
d_{k}(\zeta) \;:= \; \frac kN \;+\; {\bf 1}\{n\ge 2\} 
\sum_{r=2}^n \frac{\zeta_{z_r}}{N}\;,\quad \epsilon N 
\le k \le (1-2\epsilon) N\;.
$$
To keep notation simple let $\phi_k$ stand for $\phi(d_k(\zeta))$.  By
the Cauchy-Schwarz inequality, the last expression is less than
$\{m_\star(x)m_\star(y)\}^{\epsilon N} N^{-2\alpha} I_{\alpha}^{-2} $
times
\begin{equation*}
\sum_{m=0}^{2\ell_N} \; \sum_{\zeta \in E_{m,S\setminus\{x,y\}}} 
\!\frac{m_\star^{\zeta}}{a(\zeta)} \sum_{k= \epsilon N}^{(1-2\epsilon)N} 
\int_{\phi_{k}}^{\phi_{k+1}} u^{\alpha}(1-u)^{\alpha}\,du  
\int_{\phi_k}^{\phi_{k+1}} \frac{ u^{\alpha}(1-u)^{\alpha} }
{ (\frac kN)^{\alpha}(1-\frac{k+m}N)^{\alpha} } \,du\;.
\end{equation*}
Since $m\le 2\ell_N$ then, for all $N$ large enough, the last integral
above is less than
$$
\{\phi_{k+1}- \phi_{k}\} \big(\sup_{u\in [0,1]}
\{ \phi(u)/(u-\epsilon) \} \,\big)^{2\alpha} \;\le\; 
\frac 1N (1+\sqrt{\epsilon})^{2\alpha +1}\;.
$$
The last inequality follows from assumptions (\ref{ad1}) and
(\ref{ad2}). Therefore, we conclude that (\ref{equ3}) is bounded above
by
\begin{equation*}
\frac{\{m_\star(x)m_\star(y)\}^{\epsilon N}
(1+\sqrt{\epsilon})^{2\alpha+1}}{I_{\alpha} N^{2\alpha+1}} \; 
\sum_{m=0}^{2\ell_N} \; \sum_{\zeta \in E_{m,S\setminus\{x,y\}}} 
\frac{m_\star^{\zeta}}{a(\zeta)}\;,
\end{equation*}
which in turn is bounded by
$$
\frac{Z_{S} \Gamma(\alpha) \{m_\star(x)m_\star(y)\}^{\epsilon N}
(1+\sqrt{\epsilon})^{2\alpha+1}}{\kappa_\star \Gamma_x \Gamma_y 
I_{\alpha} N^{2\alpha+1}}\;.
$$
Hence, we have shown that (\ref{equ1}) is bounded above by
$$
m(z_i)r(z_i,z_j)\{{\bf f_{xy}}(z_i) - {\bf f_{xy}}(z_j)\}^2
\Big( \frac{Z_{S}\Gamma(\alpha)\{1+ \sqrt{\epsilon}\}^{2\alpha +1}}
{M_{\star} \kappa_\star Z_{N,S}\Gamma_x \Gamma_y I_{\alpha}} \Big) 
\{ m_\star(x) m_\star(y)\}^{\epsilon N} N^{-\alpha - 1}\;.
$$

In a similar way we can get the same upper bound for (\ref{equ1}) if
we suppose instead that $j<i$. The assertion of the proposition
follows from this estimate and Proposition \ref{zk}.
\end{proof}

We are now in a position to prove Proposition \ref{s08}.  Let
$S^1_{\star}:=S^1\cap S_\star$, $S^2_{\star}:=S^2\cap S_\star$ and
suppose they are both nonempty sets. Since
$$
\Cap_N \big( \ms E_N(S^1_\star) , \ms E_N(S^2_\star) \big) \;\le\; 
\Cap_N\Big( \,  \bigcup_{x\in S^1}D^x_N , \bigcup_{y\in S^2} D^y_N \, \Big) 
\;\le\; D_N(F_{S^1})\;,
$$
it follows from (\ref{ef1}), (\ref{est2}) and Proposition \ref{est5}
that
$$
\limsup_{N\to\infty} N^{\alpha+1} \Cap_N \big( \ms E_N(S^1_\star) , 
\ms E_N(S^2_\star) \big) \;\le\; 
\frac{ 1 }{ M_{\star}\kappa_\star I_{\alpha} \Gamma(\alpha) } 
\, \sum_{x\in S^1_\star, y\in S^2_\star} \Cap_S(x,y)\;,
$$
after letting $\epsilon \downarrow 0$. Theorem \ref{mt1} follows from
Proposition \ref{s08} and Proposition \ref{s05}.

\section{Proof of Theorem \ref{mt2}}
\label{sec5}

In \cite{bl2}, we reduced the proof of the metastability of reversible
processes to the verification of three conditions, denoted by {\bf
  (H0)}, {\bf (H1)} and {\bf (H2)}. The proof of condition {\bf (H1)}
is similar to the one presented in \cite{bl3} for zero range processes
on complete graphs. However, in the case where $m$ is not uniform,
some modifications are needed to handle sites not in $S_\star$. This
is the only reason for which we have introduced the sequences
$b_N(z)$, $z\in S\setminus S_\star$ and the respective condition in
(\ref{clb}).

The following notation will be used throughout this section. For each
$x\in S_\star$, let $\xi^x_N\in E_N$ be the configuration with $N$
particles at $x$ and let $\breve{\ms E}^x_N$ represent the set $\ms
E_N(S_\star\setminus \{x\})$.

Condition {\bf (H2)} follows immediately from (\ref{delta}) since
$\mu_N(\ms E^x_N)\to 1/\kappa_\star$ for every $x\in S_\star$:
\begin{equation}
\tag*{\bf (H2)} \lim_{N\to\infty} \frac{ \mu_N(\Delta_N) }
{ \mu_N(\ms E^x_N)} 
\;=\;0 \;,\quad \forall x\in S_\star\;.
\end{equation}

Fix a configuration $\eta$ in $\ms E^{x}_N$. Since $\sum_{y\not =
  x} \eta_y \le \ell_N$ and $\eta_z \le b_N$, for $z\in S\setminus
S_\star$, by the explicit form of $\mu_N$,
\begin{equation}
\label{f19}
\begin{split}
\mu_N(\eta) \;& \ge\; C_0 \, \prod_{z\in S\setminus S_\star} 
m_\star(z)^{\eta_z} \, \prod_{y\in S\setminus\{x\}} \frac 1{a(\eta_y)} \\
\;&\ge\; \frac {C_0}{\ell_N^{\alpha (\kappa -1)}} \, 
\prod_{z\in S\setminus S_\star} m_\star(z)^{b_N(z)} \;.
\end{split}
\end{equation}
Hereafter, $C_0$ stands for a constant which does not depend on $N\ge
1$ and whose value may change from line to line. To estimate the
capacity, $\Cap_N(\{ \eta\}, \{\xi^x_N\})$ we consider a path
$\eta^{(j)}$, $0\le j \le p$, from $\eta^{(0)}=\eta$ to
$\eta^{(p)}=\xi^x_N$ obtained by moving to $x$, one by one, each
particle. Since there are at most $\ell_N$ particles to move, we can
take a path such that $p\le \kappa \,\ell_N $. Let $F$ be an arbitrary
function in $\mc C_N(\{\eta\} , \{\xi^x_N\})$. By Cauchy-Schwarz
inequality and the explicit expression of the Dirichlet form,
\begin{equation*}
1 \;=\; \Big\{\; \sum_{j=0}^{p-1} \big[\, F(\eta^{(j+1)}) -
F(\eta^{(j)}\,)\big] \; \Big\}^2 \;\le\; C_0 \, D_N(F)\, 
\sum_{j=0}^{p-1} \frac{1}{\mu_N(\eta^{(j)}) }\;\cdot
\end{equation*}
Therefore, by \eqref{f19},
\begin{equation*}
\Cap_N (\{\eta\},\{\xi^x_N\}) \;\ge\;  
\frac {C_0}{\ell_N^{1+\alpha (\kappa -1)}} \, 
\prod_{z\in S\setminus S_\star} m_\star(z)^{b_N(z)} \;.
\end{equation*}
The extra factor $\ell_N$ comes from the length of the path. Condition
{\bf (H1)} follows now from this estimate, Theorem \ref{mt1} and
condition (\ref{clb}):
\begin{equation*}
\tag*{\bf (H1)}
\lim_{N\to\infty}\; \frac{\Cap_N(\ms E^x_N, \breve{\ms E}^x_N)}
{ \inf_{\eta\in \ms E^x_N}\{\Cap_N( \{\eta\} , \{\xi^x_N\})\} }\; 
=\; 0 \;, \quad \forall x\in S_\star \;.
\end{equation*}

Finally condition {\bf (H0)} follows from Theorem \ref{mt1} as we show
below.  Denote by $R_N^{\ms E^\star}(\cdot,\cdot)$ the jump rates of
the trace process $\{ \eta^{\ms E_N^\star}_t : t\ge 0\}$ defined in
Section \ref{sec0}. For $x \not =y$ in $S_\star$, let
\begin{equation*}
r_N(x,y) \;:=\; \frac{1}{\mu_N(\ms E^x_N)} 
\sum_{\substack{\eta\in \ms E^x_N \\ \xi \in \ms E^y_N }} \mu_N(\eta) 
\, R_N^{\ms E^\star}(\eta,\xi)\;.
\end{equation*}

By Lemma 6.8 in \cite{bl2},
\begin{equation*}
\begin{split}
& \mu_N(\ms E^x_N)\, r_N(\ms E^x_N ,\ms E^y_N) \;=\; \frac 12 
\Big\{ \, \Cap_N \big(\ms E^x_N \,,\, \breve{\ms E}^{x}_N \big) \,+\, 
\Cap_N \big(\ms E^y_N \,,\, \breve{\ms E}^{y}_N \big) \\ 
& \qquad\qquad\qquad\qquad\qquad\qquad\qquad\qquad
-\; \Cap_N \Big(\ms E_N(\{x,y\})\,,\, 
\ms E_N(S_\star\setminus\{x,y\})\Big) \,\Big\}\;.
\end{split}
\end{equation*}
Therefore, by Theorem \ref{mt1}, since $\mu_N(\ms E^x_N)$ converges to
$\kappa_\star^{-1}$ for all $x$ in $S_\star$,
\begin{equation*}
\tag*{\bf (H0)}
\lim_{N\to\infty} N^{1+\alpha} r_N(\ms E^x_N , \ms E^y_N) \;=\; 
\frac {\Cap_S(x ,y)}{M_\star  \,  
\Gamma(\alpha) \, M_\alpha}  \;, 
\quad \forall x,y\in S_\star\,,\;x\not = y\; .
\end{equation*}
This proves Theorem \ref{mt2} as a consequence of Theorem 2.10 in \cite{bl2}.

\end{document}